\documentclass{victor}

\usepackage{amstext}
\usepackage{amssymb}
\usepackage{amsfonts}
\usepackage{amsmath}
\usepackage{amsthm}

\usepackage{stmaryrd}

\begin{document}

\theoremstyle{plain}
\newtheorem{Definition}{Definition}
\newtheorem{Theorem}{Theorem}
\newtheorem{Corollary}[Theorem]{Corollary}
\newtheorem{Lemma}[Theorem]{Lemma}
\newtheorem{Proposition}[Theorem]{Proposition}

\newtheorem{theorem}[Theorem]{Theorem}
\newtheorem{definition}[Theorem]{Definition}
\newtheorem{lemma}[Theorem]{Lemma}
\newtheorem{proposition}[Theorem]{Proposition}
\newtheorem{example}[Theorem]{Example}
\newtheorem{corollary}[Theorem]{Corollary}
\newtheorem{conjecture}[Theorem]{Conjecture}
\newenvironment{note}{\noindent{\bf Note:}}{\medskip}

\def\squareforqed{\hbox{\rlap{$\sqcap$}$\sqcup$}}
\def\qed{\ifmmode\squareforqed\else{\unskip\nobreak\hfil
\penalty50\hskip1em\null\nobreak\hfil\squareforqed
\parfillskip=0pt\finalhyphendemerits=0\endgraf}\fi}

\def\cal#1{{\mathcal #1}}
\def\bb#1{{\mathbb #1}}
\def\bsigma#1{\mathbf{\Sigma}^0_{#1}}
\def\bpi#1{\mathbf{\Pi}^0_{#1}}
\def\bdelta#1{\mathbf{\Delta}^0_{#1}}
\def\analytic{\mathbf{\Sigma}^1_1}
\def\coanalytic{\mathbf{\Pi}^1_1}
\def\bianalytic{\mathbf{\Delta}^1_1}
\def\baire{\omega^{\omega}}
\def\cntbased{$\omega${\bf Top}$_0$}

\def\bdiff#1{\mathcal{D}_{#1}}
\def\mybf#1{{\mathbf #1}}

\def\cofinite{\omega_{cof}}

\def\wayabovearrow{\rlap{\raise-.25ex\hbox{$\shortuparrow$}}\raise.25ex\hbox{$\shortuparrow$}}
\def\waybelowarrow{\rlap{\raise.25ex\hbox{$\shortdownarrow$}}\raise-.25ex\hbox{$\shortdownarrow$}}

\def\abs#1{{\lvert#1\rvert}}



\title{Levels of discontinuity, limit-computability, and jump operators}

\author{Matthew de Brecht\thanks{This paper is dedicated to Victor Selivanov in celebration of his 60th birthday and his valuable contributions to the descriptive set theory of general topological spaces. The author thanks Arno Pauly, Luca Motto Ros, Vasco Brattka, and Takayuki Kihara for valuable discussions and comments on earlier drafts of this paper.}}

\runningtitle{Levels of discontinuity and jump operators}
\runningauthor{Matthew de Brecht}


\altmaketitle

\begin{abstract}
We develop a general theory of jump operators, which is intended to provide an abstraction of the notion of ``limit-computability'' on represented spaces. Jump operators also provide a framework with a strong categorical flavor for investigating degrees of discontinuity of functions and hierarchies of sets on represented spaces. We will provide a thorough investigation within this framework of a hierarchy of ${\bf \Delta}^0_2$-measurable functions between arbitrary countably based $T_0$-spaces, which captures the notion of computing with ordinal mind-change bounds. Our abstract approach not only raises new questions but also sheds new light on previous results. For example, we introduce a notion of ``higher order'' descriptive set theoretical objects, we generalize a recent characterization of the computability theoretic notion of ``lowness'' in terms of adjoint functors, and we show that our framework encompasses ordinal quantifications of the non-constructiveness of Hilbert's finite basis theorem.  
\end{abstract}



\section{Introduction}

This paper is concerned with two relatively new developments in the field of descriptive set theory. 

The first development is the extension of the classical descriptive set theory for metrizable spaces to more general topological spaces and mathematical structures. Although it is not uncommon, particularly in measure theory, to define the Borel algebra for an arbitrary topological space, detailed analysis of the Borel hierarchy has been mainly restricted to the class of metrizable spaces, or possibly Hausdorff spaces on rare occasion. However, relatively recent work by V. Selivanov \cite{selivanov, selivanov1984,selivanov2004,selivanov2007,selivanov2008}, D. Scott \cite{scott_datatypes}, A. Tang \cite{tang1979,tang1981}, and the author \cite{debrecht2013,debrecht_etal_4}, have demonstrated that a significant portion of the descriptive set theory of metrizable spaces generalizes naturally to countably based $T_0$-spaces. This development opens up the possibility of finding new applications of descriptive set theory to mathematical fields heavily relying on non-Hausdorff topological spaces, such as theoretical computer science (e.g., $\omega$-continuous domains) and modern algebraic geometry (e.g., the Zarisiki topology on the prime spectrum of a countable commutative ring). These generalizations can also shed new light on old results. For example, although the Gandy-Harrington space (a non-metrizable space that plays an important role in effective descriptive set theory) cannot be topologically embedded into any Polish space, it can be embedded as a co-analytic set into a quasi-Polish space \cite{debrecht2013}.

The second development is a shift from a focus on the complexity of subsets of a space to a focus on the complexity of functions between spaces. Certainly Baire's hierarchy of discontinuous functions has a long history, but it is fair to say that Borel's hierarchy of sets has played a more prominent role in the development of the theory. However, recently there has been growing interest within the field of computable analysis concerning the relationship between hierarchies of discontinuous functions, Turing degrees, and limit-computability, in particular by researchers such as V. Brattka, P. Hertling, A. Pauly, M. Ziegler, T. Kihara, and the author \cite{brattka, brattka_makananise, hertling1996, brattka_debrecht_pauly, ziegler,kihara2013, debrecht_etal_4}. Furthermore, recent extensions of the Wadge game by researchers such as A. Andretta, L. Motto Ros, and B. Semmes \cite{andretta,ros_semmes,ros_thesis,ros2011, semmes_thesis} have provided new classifications of discontinuous functions and new methods to generalize classical results like the Jayne-Rogers theorem \cite{jayne_rogers}. V. Selivanov has made contributions in this area as well, for example by generalizing the Hausdorff-Kuratowski theorem for the difference hierarchy to a hierarchy of $\bdelta 2$-measurable functions into finite discrete spaces \cite{selivanov2007,selivanov2008}. 

These two developments should not be considered independent. For example, if we simply add to our framework the two-point non-trivial \emph{non-metrizable} space $\cal S$, known as the Sierpinski space, then we can obtain an elegant bijective correspondence between the family of $\bsigma \alpha$-subsets of a space $X$ and the family of $\bsigma \alpha$-measurable functions from $X$ to $\cal S$. This is a natural generalization of the known bijection between open subsets of a space and continuous functions into the Sierpinski space, and is also similar to the role of the subobject classifier in a topos. Domain theory teaches us that the mathematical object $\bsigma \alpha(X)$, now viewed as a family of functions into $\cal S$, will certainly not be metrizable, even if we can hope for it to be a topological space at all.

Continuing a little more our analogy with a topos, if we wish to work within a single category then we are faced with a dilemma if we have only one ``sub-object'' classifier $\cal S$ but want a \emph{hierarchy} of classes of subobjects such as $\bsigma 1(X)\subseteq \bsigma 2(X)\subseteq \cdots$. One natural solution is to abandon the idea of having a single subobject classifier, and instead have a sequence $\cal S_1,\cal S_2,\ldots$ of subobject classifiers that respectively classifiy the families $\bsigma 1, \bsigma 2, \ldots$. Such a theory would be very unwieldy if the subobject classifiers were all unrelated, but we might have some hope for the theory if the subobject classifiers $\cal S_1,\cal S_2,\ldots$ are defined as the iterates of a single endofunctor $F$ applied to a single subobject classifier $\cal S$. We now only have to worry about which functors $F$ to consider, and what the ``base'' subobject classifier $\cal S$ should be.

This abstract view is closely related to recent work initiated by A. Pauly on synthetic descriptive set theory \cite{pauly2012talk, pauly_debrecht2013}. Ultimately, an axiomatic approach in the same spirit as topos theory would be most desireable, as it might help expose connections between the descriptive set theory of general topological spaces and the descriptive complexity of finite structures \cite{addison}. However, it seems a little premature to attempt that now, and instead we develop these ideas within the category of represented spaces and continuously realizable functions \cite{bauer}. In this context, we introduce (topological) ``jump operators'', which modulate the representation of a space and in effect play the role of the endofunctors $F$ described above.

The concept of a jump operator that we present here has its roots in the work of M. Ziegler \cite{ziegler} and V. Brattka \cite{brattka,brattka_makananise}, where numerous connections are made between levels of discontinuity, limit-computability, and the representation of a function's output space. The hierarchy of discontinuity that jump operators characterize turns out to be a subset of the strong Weihrauch degrees \cite{brattka_gherardi,brattka_gherardi2}, but we believe that the categorical framework that jump operators provide has much to offer.

This paper is organized into five major sections. After this Introduction, we will develop the general theory of (topological) jump operators. The third major section will investigate a lower portion of the jump operator hierarchy consisting of $\bdelta 2$-measurable functions. Our main contribution here is to extend some previous results concerning functions between metrizable spaces to functions between arbitrary countably based $T_0$-spaces. The results in this section are also important because they demonstrate that the jump operator framework is powerful enough to characterize functions as finely as P. Hertling's hierarchy of discontinuity levels \cite{hertling_thesis,hertling1996}. The fourth major section presents several examples and applications, such as connections with the difference hierarchy, a quantification of the non-constructiveness of Hilbert's basis theorem in terms of the ordinal $\omega^\omega$ (essentially due to S. Simpson \cite{simpson1988} and F. Stephan and Y. Ventson \cite{stephan_ventsov}), and show some applications to the Jayne-Rogers theorem. It is our attempt to find a common thread between the results in this section that should be considered new, more so than the results themselves, so in several cases we omit proofs. We conclude in the fifth major section. 

We will expect that the reader is familiar with classical descriptive set theory \cite{kechris} and domain theory \cite{etal_scott}. The reader should also consult \cite{selivanov} and \cite{debrecht2013} for definitions and results concerning the descriptive set theory of arbitrary countably based $T_0$-spaces. Although we will not be concerned much with computability issues, the reader will benefit from an understanding of the Type Two Theory of Effectivity \cite{weihrauch}. In particular, we will make much use of M. Schr\"{o}der's extended definition of an admissible representation \cite{schroder}, as well as the notion of realizability of functions between represented spaces (see \cite{bauer}). 

Our notation will follow that of \cite{debrecht2013}. The following modification of the Borel hierarchy, due to V. Selivanov, is required in order to provide a meaningful classification of the Borel subsets of non-metrizable spaces.

\begin{definition}
Let $X$ be a topological space. For each ordinal $\alpha$ ($1\leq \alpha < \omega_1$) we define $\bsigma \alpha(X)$ inductively as follows.
\begin{enumerate}
\item
$\bsigma 1(X)$ is the set of all open subsets of $X$.
\item
For $\alpha>1$, $\bsigma \alpha(X)$ is the set of all subsets $A$ of $X$ which can be expressed in the form
\[A=\bigcup_{i\in\omega}B_i\setminus B'_i,\]
where for each $i$, $B_i$ and $B'_i$ are in $\bsigma {\beta_i}(X)$ for some $\beta_i<\alpha$.
\end{enumerate}
We define $\bpi \alpha(X)= \{X\setminus A\,|\, A\in \bsigma \alpha(X)\}$ and $\bdelta \alpha(X)=\bsigma \alpha(X)\cap \bpi \alpha(X)$.
\qed
\end{definition}

The above definition is equivalent to the classical definition of the Borel hierarchy for metrizable spaces, but it differs for more general spaces. 

A function $f\colon X\to Y$ is $\bsigma \alpha$-measurable if and only if $f^{-1}(U)\in\bsigma \alpha(X)$ for every open subset $U$ of $Y$. We will also be interested in $\bdelta 2$-measurability, which requires the preimage of every open set to be a $\bdelta 2$-set.

Later in the paper we will present some results specific to \emph{quasi-Polish} spaces, which are defined as the countably based spaces that admit a Smyth-complete quasi-metric. Polish spaces and $\omega$-continuous domains are examples of quasi-Polish spaces. A space is quasi-Polish if and only if it is homeomorphic to a $\bpi 2$-subset of $\cal P(\omega)$, the power set of $\omega$ with the Scott-topology. The reader should consult \cite{debrecht2013} for additional results on quasi-Polish spaces.

\section{Jump operators}

A \emph{represented space} is a pair $\langle X, \rho\rangle$ where $X$ is a set and $\rho \colon \subseteq \baire \to X$ is a surjective partial function. If $\langle X, \rho_X\rangle$ and $\langle Y, \rho_Y\rangle$ are represented spaces and $f\colon\subseteq X\to Y$ is a partial function, then a function $F\colon\subseteq \baire\to\baire$ \emph{realizes} $f$, denoted $F\vdash f$, if and only if $f\circ\rho_X = \rho_Y\circ F$. If there exists a continuous realizer for $f$ then we say that $f$ is \emph{continuously realizable} and write $\vdash f$. Note that if $F \vdash f$ and $G\vdash g$, then $G\circ F \vdash g\circ f$, assuming the composition $g\circ f$ makes sense.

In some cases, a function $f\colon X\to Y$ between represented spaces may fail to be continuously realizable, but will become continuously realizable if we strengthen the information content of the representation of $X$ or weaken the information content of the representation of $Y$. The notion of ``limit computability'' is a common example of weakening the output representation. The motivation for the following definition is to create an abstract framework to investigate in a uniform manner how modifications of the information content of a representation effects the realizability of functions. 

\begin{definition}
A \emph{(topological) jump operator} is a partial surjective function $j\colon\subseteq \baire\to\baire$ such that for every partial continuous $F\colon \subseteq \baire\to\baire$, there is partial continuous $F' \colon \subseteq \baire\to\baire$ such that $F\circ j = j\circ F'$.
\qed
\end{definition}

The identity function $id \colon\baire\to\baire$ is a trivial example of a jump operator. Let $f\colon X\to Y$ be a function between represented spaces. A \emph{$j$-realizer} of $f$ is a function $F\colon \subseteq \baire \to \baire$ such that  $j \circ F \vdash f$. We use the notation $F\vdash_j f$ to denote that $F$ is a $j$-realizer for $f$. If there exists a continuous $j$-realizer for $f$ then we will say that $f$ is \emph{$j$-realizable} and write $\vdash_j f$.

The definition of ``jump operator'' given above is appropriate for the category of represented spaces and continuously realizable functions. Given a represented space $\langle X, \rho_X\rangle$ and a jump operator $j$, we can write $j(X)$ to denote the represented space $\langle X, \rho_X\circ j\rangle$. For each function $f\colon X\to Y$ between represented spaces, we define $j(f)$ to be the same function as $f$ but now interpretted as being between the represented spaces $j(X)$ and $j(Y)$. It is now clear that the definition of a jump operator is precisely what is needed to guarantee that $j(\cdot)$ is a well-defined endofunctor on the category of represented spaces.

If working in the category of represented spaces and \emph{computably} realizable functions, then the appropriate definition of a (computability theoretic) jump operator would be to require that for every \emph{computable} $F\colon \subseteq \baire\to\baire$, there is \emph{computable} $F'\colon \subseteq \baire\to\baire$ such that $F\circ j = j\circ F'$. The definition of $j$-realizability would also be modified in a similar manner. These modifications are necessary because, for example, the operator $\mathfrak{L}$ introduced in \cite{brattka_debrecht_pauly} to characterize low computability is a computability theoretic jump operator but it is \emph{not} a topological jump operator.

In this paper, unless explicitly mentioned otherwise we will assume the topological jump operator definition given above and shall drop the term ``topological''. However, much of the theory we develop will also apply to the computability theoretic jump operators as well.

Examples \ref{ex:jsigma2}, \ref{ex:jdelta}, and \ref{ex:jalpha} below provide typical examples of jump operators. In the following, $\langle \cdots \rangle_{n\in\omega}\colon (\baire)^\omega \to \baire$ is some fixed (computable) encoding of countable sequences of elements of $\baire$ as single elements of $\baire$.

\begin{example}\label{ex:jsigma2}
Define $j_{\bsigma 2}\colon \subseteq \baire\to\baire$ as:
\begin{eqnarray*}
\langle \xi_n \rangle_{n\in\omega}\in dom(j_{\bsigma 2}) &\Leftrightarrow& \xi_0, \xi_1, \ldots \mbox{ converges in } \baire\\
j_{\bsigma 2}(\langle \xi_n \rangle_{n\in\omega}) &=& \lim_{n\in\omega} \xi_n
\end{eqnarray*}
\qed
\end{example}

\begin{example}\label{ex:jdelta}
Define $j_{\Delta}\colon \subseteq \baire\to\baire$ as:
\begin{eqnarray*}
\langle \xi_n \rangle_{n\in\omega}\in dom(j_{\Delta}) &\Leftrightarrow&  (\exists n)(\forall m\geq n) [\xi_m=\xi_n] \\
j_{\Delta}(\langle \xi_n \rangle_{n\in\omega}) &=& \lim_{n\in\omega} \xi_n
\end{eqnarray*}
\qed
\end{example}

\begin{example}\label{ex:jalpha}
For each countable ordinal $\alpha$, define $j_{\alpha}\colon \subseteq \baire\to\baire$ as:
\begin{eqnarray*}
\langle \langle \beta_n \rangle_{\alpha}\diamond \xi_n \rangle_{n\in\omega} \in dom(j_{\alpha}) &\Leftrightarrow& (\forall n)  (\alpha>\beta_n \geq \beta_{n+1}) \mbox{ and }\\
&&  (\forall n) (\xi_n \not= \xi_{n+1}\Rightarrow \beta_n \not= \beta_{n+1}) \\
j_{\alpha}(\langle \langle \beta_n \rangle_{\alpha}\diamond \xi_n \rangle_{n\in\omega}) &=& \lim_{n\in\omega} \xi_n
\end{eqnarray*}
where $\langle \cdot \rangle_{\alpha}\colon \alpha \to \omega$ is some fixed encoding of ordinals less than $\alpha$ as natural numbers, and $\langle \beta \rangle_{\alpha}\diamond \xi$ is the element of $\baire$ obtained by prepending the encoding of $\beta$ to the beginning of $\xi$. 
\qed
\end{example}

The jump operators $j_{\bsigma 2}$ and $j_{\Delta}$ are also computability theoretic jump operators. If $\alpha<\omega^{CK}_1$ then $j_{\alpha}$ is a computability theoretic jump operator assuming that the encoding $\langle \cdot\rangle_{\alpha}$ is effective.

Intuitively, $j_{\bsigma 2}$-realizing a function only requires the realizer to output a sequence of ``guesses'' which is guaranteed to converge to the correct answer. Each guess is an infinite sequence in $\baire$, and convergence means that each finite prefix of the guess can be modified only a finite number of times. The jump operator $j_{\bsigma 2}$ and its connections with limit computability have been extensively studied in the field of computable analysis, for example by V. Brattka \cite{brattka, brattka_makananise} and M. Ziegler \cite{ziegler}. In the context of Wadge-like games and reducibilities, the jump operator $j_{\bsigma 2}$ essentially captures the notion of an ``eraser'' game (see \cite{ros_thesis, ros2011,semmes_thesis} and the references therein). It can also be shown that $j_{\bsigma 2}$ is a $\bsigma 2$-admissible representation of $\baire$ in the sense of \cite{debrecht_etal_4}. It follows from these results that if $X$ and $Y$ are countably based $T_0$-spaces, then $f\colon X\to Y$ is $j_{\bsigma 2}$-realizable if and only if $f$ is $\bsigma 2$-measurable. In general, for any ordinal $\alpha < \omega_1$,  any $\bsigma \alpha$-admissible representation $j_{\bsigma \alpha}$ of $\baire$ will be a jump operator which precisely captures the $\bsigma \alpha$-measurable functions between countably based spaces. For finite $n>2$ the inductive definition $j_{\bsigma {n+1}} = j_{\bsigma n}\circ j_{\bsigma 2}$ suffices. 

The jump operator $j_{\Delta}$ defines a stricter notion of limit computability. In this case, the realizer is also allowed to output guesses that converge to the correct answer, but the realizer can only modify his guess a finite number of times. This jump operator has been investigated by M. Ziegler \cite{ziegler} in terms of finite revising computation, and was shown in \cite{brattka_debrecht_pauly} by V. Brattka, A. Pauly, and the author to correspond to closed choice on the natural numbers. In the context of Wadge-like games, $j_{\Delta}$ corresponds to the ``backtrack'' game (see \cite{semmes_thesis}). A. Andretta \cite{andretta} has shown that a total function on $\baire$ is $j_{\Delta}$-realizable if and only if it is $\bdelta 2$-piecewise continuous. It follows from the Jayne-Rogers theorem (\cite{jayne_rogers}, see also \cite{ros_semmes,kacena_etal,solecki1998,kihara2013}) that a total function on $\baire$ is $j_{\Delta}$-realizable if and only if it is $\bdelta 2$-measurable. However, it should be noted that this relationship between $j_\Delta$-realizability and $\bdelta 2$-measurability does not extend to functions between arbitrary spaces, and in fact we will show in the latter half of this paper that there is no jump operator that completely captures the notion of $\bdelta 2$-measurability.

The family of jump operators $j_\alpha$ for $\alpha<\omega_1$ are further restrictions of $j_{\Delta}$, where the realizer must output with each guess an ordinal bound on the number of times it will change its guess in the future. For example, when $1\leq n <\omega$, a $j_n$-realizer can only make a maximum of $n$ guesses. A $j_\omega$-realizer must output a bound $n<\omega$ along with its first guess, and then can only modify its guess a maximum of $n$ times thereafter. This concept is closely related to the Ershov hierarchy \cite{ershov}, to the notion of ordinal mind change complexity in the field of inductive inference (see \cite{freiv:smith,luo_schulte,debrecht_etal_2, debrecht_etal_3}), and to the Hausdorff difference hierarchy \cite{kechris}. We will show later in this paper that a function between countably based $T_0$-spaces is $j_\alpha$-realizable if and only if the discontinuity level of the function (in the sense of P. Hertling \cite{hertling_thesis,hertling1996}) does not exceed $\alpha$.

\subsection{Lattice structure}

Jump operators are (quasi-)ordered by $j \leq k$ if and only if $j$ is $k$-realizable. We will say $j$ and $k$ are equivalent, written $j\equiv k$, if  $j\leq k$ and $k\leq j$. In the examples given previously, it is clear that $j_\alpha \leq j_\beta \leq j_\Delta \leq j_{\bsigma 2}$ when $\alpha \leq \beta < \omega_1$. It is straightforward to prove that if $j\leq k$, then $\vdash_j f$ implies $\vdash_k f$ for every function $f$ between represented spaces.

In this section we will prove that the topological jump operators form a lattice which is complete under countable (non-empty) meets and joins. The definitions and main points of the proofs in this section are mostly due to A. Pauly \cite{pauly2010}. The author is indebted to A. Pauly for pointing out that the proofs in this section only apply to topological jump operators and may fail to hold for computability theoretic jump operators in general.

In the following, given $i\in\omega$ and $\xi\in\baire$, we will write $\langle i\rangle\diamond \xi$ to denote the element of $\baire$ obtained by prepending $i$ to the beginning of $\xi$.

\begin{definition}
Let $(j_i)_{i\in\omega}$ be a countable sequence of jump operators. Define $\bigvee j_i \colon \subseteq \baire \to \baire$ and $\bigwedge j_i \colon \subseteq \baire \to \baire$ by 
\begin{enumerate}
\item
$(\bigvee j_i)(\langle i \rangle \diamond \xi) = j_i(\xi)$, where
\[dom(\bigvee j_i) = \{\langle i\rangle \diamond \xi \in \baire \,|\, \xi \in dom(j_i)\}.\]
\item
$(\bigwedge j_i)(\langle \xi_n \rangle_{n\in\omega}) = j_0(\xi_0)$, where
\[ dom(\bigwedge j_i) = \{\langle \xi_n \rangle_{n\in\omega} \in \baire \,|\, \forall i,k\colon j_i(\xi_i)=j_k(\xi_k)\}.\]
\end{enumerate}
\qed
\end{definition}

The next two theorems show that the above definitions are in fact (topological) jump operators corresponding to the supremum and infimum of $(j_i)_{i\in\omega}$. 

\begin{theorem}
$\bigvee j_i$ is a jump-operator and is the supremum of $(j_i)_{i\in\omega}$.
\end{theorem}
\begin{proof}
Assume $f\colon \subseteq \baire\to\baire$ is continuous. For $i\in \omega$ there is continuous $g_i\colon \subseteq \baire\to\baire$ such that $f\circ j_i = j_i\circ g_i$. Define $g\colon \subseteq \baire\to \baire$ as $g(\langle i \rangle \diamond \xi) = \langle i \rangle \diamond g_i(\xi)$. Clearly $g$ is continuous and
\begin{eqnarray*}
f((\bigvee j_i)(\langle i \rangle \diamond \xi)) &=& f(j_i(\xi))\\
			&=& j_i(g_i(\xi))\\
			&=& (\bigvee j_i)(\langle i \rangle \diamond g_i(\xi))\\
			&=& (\bigvee j_i)(g(\langle i \rangle \diamond \xi)),
\end{eqnarray*}
hence $f\circ \bigvee j_i = \bigvee j_i \circ g$. Therefore $\bigvee j_i$ is a jump operator.

Next, for $i\in \omega$ define $f_i(\xi)=\langle i \rangle \diamond \xi$. Then $f_i$ is continuous and $j_i = (\bigvee j_i)\circ f_i$. Therefore $j_i \leq \bigvee j_i$ for all $i\in \omega$.

Finally, assume $p\colon \baire\to\baire$ is such that $j_i = p \circ q_i$ for some continuous $q_i\colon \baire\to\baire$ (for all $i\in\omega$). Define $q\colon \baire\to\baire$ so that $q(\langle i \rangle \diamond \xi) = q_i(\xi)$. Then $q$ is continuous and 
\begin{eqnarray*}
(\bigvee j_i)(\langle i \rangle \diamond \xi) &=&j_i(\xi)\\
 &=& p(q_i(\xi))\\
 &=& p(q(\langle i \rangle \diamond \xi))
\end{eqnarray*}
hence $\bigvee j_i = p\circ q$. Therefore $\bigvee j_i\leq p$. It follows that $\bigvee j_i$ is the supremum of $(j_i)_{i\in\omega}$.
\end{proof}

\begin{theorem}
$\bigwedge j_i$ is a jump-operator and is the infimum of $(j_i)_{i\in\omega}$.
\end{theorem}
\begin{proof}
Assume $f\colon \subseteq \baire\to\baire$ is continuous. For $i\in \omega$ there is continuous $g_i\colon \subseteq \baire\to\baire$ such that $f\circ j_i = j_i\circ g_i$. Define $g\colon \subseteq \baire\to \baire$ as 
\[g(\langle \xi_n \rangle_{n\in\omega})=\langle g_n(\xi_n)\rangle_{n\in\omega}.\]
Clearly $g$ is continuous and if $\langle \xi_n\rangle_{n\in\omega} \in dom(\bigwedge j_i)$ then for all $i,k\in\omega$, $j_i(\xi_i)=j_k(\xi_k)$ hence $j_i\circ g_i(\xi_i) = f\circ j_i(\xi_i) = f\circ j_k(\xi_k) = j_k\circ g_k(\xi_k)$, and it follows that $\langle g_n(\xi_n) \rangle_{n\in\omega}\in dom(\bigwedge j_i)$. So for $\langle \xi_n \rangle_{n\in\omega} \in dom(\bigwedge j_i)$ we have
\begin{eqnarray*}
f((\bigwedge j_i)(\langle \xi_n \rangle_{n\in\omega})) &=& f(j_0(\xi_0))\\
  &=& j_0(g_0(\xi_0))\\
  &=& (\bigwedge j_i)(\langle g_n(\xi_n) \rangle_{n\in\omega}),
\end{eqnarray*}
hence $f\circ \bigwedge j_i = \bigwedge j_i \circ g$. Therefore $\bigwedge j_i$ is a jump operator.

Next, define $\pi_i(\langle \xi_n \rangle_n\in\omega)=\xi_i$, which is clearly continuous. Then 
\begin{eqnarray*}
(\bigwedge j_i)(\langle \xi_n \rangle_{n\in\omega}) &=& j_0(\xi_0)\\
&=& j_i(\xi_i)\\
&=& j_i(\pi_i(\langle \xi_n\rangle_{n\in\omega})),
\end{eqnarray*}
hence $\bigwedge j_i \leq j_i$ for all $i\in\omega$.

Assume $p\colon \baire\to\baire$ is such that $p = j_i \circ q_i$ for some continuous $q_i\colon \baire\to\baire$ (for all $i\in\omega$). Define $q\colon \baire\to\baire$ so that $q(\xi) = \langle q_n(\xi) \rangle_{n\in\omega}$. Clearly $q$ is continuous. If $\xi\in dom(q)$ then $p(\xi)=j_i(q_i(\xi))$ for all $i\in \omega$, so $\langle q_n(\xi) \rangle_{n\in\omega}\in dom(\bigwedge j_i)$ and 
\begin{eqnarray*}
p(\xi) &=& j_0(q_0(\xi))\\
       &=& (\bigwedge j_i)(\langle q_n(\xi) \rangle_{n\in\omega})\\
	&=& (\bigwedge j_i)(q(\xi)),
\end{eqnarray*}
hence $p\leq \bigwedge j_i$. It follows that $\bigwedge j_i$ is the infimum of $(j_i)_{i\in\omega}$.
\end{proof}

For example, let $(j_{\alpha_i})_{i\in\omega}$ be a sequence of jump operators from Example \ref{ex:jalpha}. Then it is straightforward to verify that $\bigvee j_{\alpha_i} = j_{\bigvee \alpha_i}$ and $\bigwedge j_{\alpha_i} = j_{\bigwedge \alpha_i}$.

\subsection{The ``jump'' of a representation}

Let $\cal J$ be the lattice of jump operators. It is easy to show that if $j$ and $k$ are jump operators, then so is $j\circ k$. Furthermore, if $j_1 \leq j_2$, then $j_1 \circ k \leq j_2 \circ k$. Thus, every jump operator $k$ defines a monotonic function on $\cal J$, which we call the \emph{$k$-jump}, that maps $j$ to $j \circ k$. This notion of iterating ``jumps'' can be found in \cite{ziegler} and \cite{brattka_makananise} for the case of $j_{\bsigma 2}$.

A jump operator $j$ is \emph{extensive} if the identity function $id\colon\baire\to\baire$ is $j$-realizable. Currently the author is unaware of any topological jump operators that are \emph{not} extensive, but the non-extensive computability theoretic jump operators have a non-trivial structure (for example, the inverse of the Turing jump, or integral, in \cite{brattka_makananise} and \cite{brattka_debrecht_pauly} is non-extensive).

A jump operator $j$ is \emph{idempotent} if $j\circ j \equiv j$. The jump operator $j_{\Delta}$ is idempotent, but $j_{\bsigma 2}$ is not.

We will say that $j$-realizability is \emph{closed under compositions} if for every pair of $j$-realizable functions $f\colon X \to Y$ and $g\colon Y \to Z$ we have that $g\circ f$ is also $j$-realizable.

\begin{theorem}
If $j$ is an extensive jump operator, then $j$-realizability is closed under composition if and only if $j$ is idempotent.
\end{theorem}
\begin{proof}
Assume $j$ is extensive and closed under compositions. Clearly, $j$ is $j$-realizable, so $j\circ j$ must be $j$-realizable, hence $j\circ j \leq j$. On the other hand, since $id \leq j$ it follows by the monotonicity of the $j$-jump that $j \leq j\circ j$. Therefore, $j\circ j \equiv j$.

For the converse, assume $j$ is extensive and idempotent, and assume $F\vdash_j f$ and $G\vdash_j g$ and the composition $g\circ f$ is possible. Then $j\circ F \vdash f$ and $j\circ G \vdash g$, and composition gives $j\circ G \circ j\circ F \vdash g\circ f$. Since $j$ is a jump operator, there is continuous $G'$ such that $j\circ j\circ G'\circ F \vdash g\circ f$. Now using the idempotent property of $j$ we obtain $G'\circ F \vdash_j g\circ f$. 
\end{proof}

Recall that a closure operator on a partially ordered set is a function which is monotonic, extensive, and idempotent. The above theorem can be reworded as follows.

\begin{corollary}
If $j$ is an extensive jump operator, then $j$-realizability is closed under composition if and only if the $j$-jump is a closure operator on $\cal J$.
\qed
\end{corollary}

It is easy to see that if the $j$-jump is a closure operator on $\cal J$, then $j$ is the least fixed point of the $j$-jump above $id$. In particular, the $j_{\Delta}$-jump of $j_\alpha$ is equivalent to $j_{\Delta}$ for each $\alpha<\omega_1$. It turns out that $j_{\bsigma \alpha}$ is a fixed point of the $j_{\Delta}$-jump for each $\alpha <\omega_1$ because $j_{\bsigma \alpha} \circ j_{\Delta}$ is $\bsigma \alpha$-measurable.

\subsection{Strong Weihrauch Degrees}

In this section we will compare jump operators with the notion of strong Weihrauch reducibility (see \cite{brattka_gherardi,brattka_gherardi2,pauly2010,brattka_debrecht_pauly} for more on Weihrauch reducibility). We only consider the topological version of reducibility for the case of single valued functions.

\begin{definition}
Let $f\colon X \to Y$ and $g\colon W \to Z$ be functions between represented spaces. Define $f\leq_{sW} g$ if and only if there are continuous functions $K,H\colon\subseteq \baire\to\baire$ satisfying $K\circ G \circ H \vdash f$ whenever $G\vdash g$.
\qed
\end{definition}

\begin{theorem}
Let $f \colon X \to Y$ be a function between represented spaces, and let $j$ be a jump operator. Then $f \leq_{sW} j$ if and only if $f$ is $j$-realizable.
\end{theorem}
\begin{proof}
Assume $f \leq_{sW} j$ and let $K$ and $H$ be the relevant continuous functions. Since $\baire$ is represented by the identity function, it follows that $K \circ j \circ H \vdash f $. Using the fact that $j$ is a jump operator, there is continuous $K' \colon \subseteq \baire\to\baire$ such that $j \circ K' \circ H \vdash f$. Therefore, the continuous function $K'\circ H$ $j$-realizes $f$.
  
For the converse, assume $F\colon\baire\to\baire$ is a continuous $j$-realizer of $f$. Then $j \circ F \vdash f$ by definition. Again, since $\baire$ is represented by the identity function we have $J \vdash j$ if and only if $J=j$. Thus, taking $K$ as the identity function and $H=F$ demonstrates that $f \leq_{sW} j$.
\end{proof}

The above theorem shows that jump operators form a subset of the strong Weihrauch degrees. However, this inclusion is strict, in the sense that there are strong Weihrauch degrees that do not correspond to any jump operator. For example, a constant function on $\baire$ is not strong Weihrauch equivalent to any jump operator because jump operators are surjective.

\subsection{Adjoints}

This section actually applies more to computability theoretic jump operators than topological jump operators, but the basic definitions and immediate results are the same in both cases. This section mainly consists of generalizations of results found in \cite{brattka_makananise} and \cite{brattka_debrecht_pauly}.

Let $j$ and $k$ be jump operators and let $id\colon \baire\to\baire$ be the identity function. We say that \emph{$j$ is left adjoint to $k$} or that \emph{$k$ is right adjoint to $j$}, and write $j \dashv k$, if and only if $k\circ j \leq id \leq j\circ k$. This is equivalent to stating that the $j$-jump on $\cal J$ is left adjoint to the $k$-jump, and it also implies that the associated endofunctors are adjoint.

\begin{example}[see \cite{brattka_makananise} and \cite{brattka_debrecht_pauly}]
Let $(U_n)_{n\in\omega}$ be a standard enumeration of the computably enumerable open subsets of $\baire$. Define $J\colon \baire\to\baire$ by $J(\xi)(n) = 1$ if $\xi\in U_n$ and $J(\xi)(n)=0$, otherwise. Then $J^{-1}$, the inverse of $J$, is a computability theoretic jump operator and $J^{-1} \dashv j_{\bsigma 2}$.
\qed
\end{example}

Note, however, that $J^{-1}$ is \emph{not} a topological jump operator \cite{brattka_debrecht_pauly}.

\begin{proposition}
If $j\dashv k$ then the $(j\circ k)$-jump is a closure operator. In particular, $(j\circ k)$-realizable functions are closed under composition.
\end{proposition}
\begin{proof}
This is a well known property of adjoints. Since $k\circ j \leq id$ it follows that $j\circ k\circ j\circ k\leq j\circ id \circ k\equiv j\circ k$. Furthermore, $id \leq j\circ k$ implies $j\circ k \equiv id\circ j\circ k \leq j\circ k\circ j\circ k$, and it follows that $(j\circ k)$ is idempotent. Therefore, the $(j\circ k)$-jump is a closure operator. 
\end{proof}

The \emph{low-jump-operator} is defined as ${\mathfrak{L}} = J^{-1} \circ j_{\bsigma 2}$. It is shown in \cite{brattka_debrecht_pauly} that $\mathfrak{L}$-realizability captures the notion of ``lowness'' from computability theory. It immediately follows from the above proposition that $\mathfrak{L}$-realizable functions are closed under composition.

The general theory of adjoints provides much information about $j$ and $k$ when it is known that $j\dashv k$. For example, the $j$-jump preserves joins on $\cal J$ and the $k$-jump preserves meets. Viewed as functors, $j$ preserves colimits and $k$ preserves limits. This means, in particular, that $k(X)\times k(Y)$ will be isomorphic to $k(X\times Y)$ for every pair of represented spaces $X$ and $Y$.

Although so far we have been investigating the effects of weakening the output representation, it is also interesting to investigate the effects of strengthening the input representation. Given jump operators $j$ and $k$, represented spaces $\langle X, \rho_X\rangle$ and $\langle Y, \rho_Y\rangle$, and a function $f\colon X\to Y$, we will say that a function $F\colon\subseteq \baire\to\baire$ \emph{$\langle j,k\rangle$-realizes} $f$ if and only if $f\circ\rho_X\circ j = \rho_Y\circ k \circ F$. This simply means that $F$ realizes $f$ reinterpretted as a function between $j(X)$ and $k(Y)$. We will say that a function is $\langle j,k \rangle$-realizable if and only if it has a continuous $\langle j,k\rangle$-realizer. Clearly, $j$-realizability as defined earlier corresponds to $\langle id, j\rangle$-realizability.

The following theorem shows that if $j\dashv k$, then strengthening the input representation by $j$ is equivalent to weakening the output representation by $k$. 

\begin{theorem}
If $j$ and $k$ are jump operators and $j\dashv k$, then $\langle j, id\rangle$-realizability is equivalent to $\langle id, k\rangle$-realizability.
\end{theorem}
\begin{proof}
Assume $j\dashv k$ and that $f\colon X\to Y$ is $\langle j,id\rangle$-realizable. Let $F_j$ be any continuous $\langle j,id\rangle$-realizer for $f$. Since $k$ is a jump operator there is a partial continuous $F'_j$ that $\langle k, k \rangle$-realizes $F_j$, hence $F'_j$ is a $\langle j\circ k, k\rangle$-realizer of $f$. If we let $I$ be a continuous function reducing $id$ to $j\circ k$, then $F'_j\circ I$ is a continuous $\langle id, k\rangle$-realizer for $f$. Therefore, $f$ is $\langle id, k\rangle$-realizable.

Proving that $\langle id, k\rangle$-realizability implies $\langle j, id \rangle$-realizability is done similarly.
\end{proof}

Finally, the following proposition shows that it is easy to create new pairs of adjoint jump operators from a given pair of adjoint operators. We leave the proof as an easy exercise.

\begin{proposition}
If $j\dashv k$, then $k\circ k\circ j \circ j \leq k\circ j \leq id \leq j\circ k \leq j\circ j\circ k\circ k$. In particular, we have $j\circ j \dashv k\circ k$.
\qed
\end{proposition}

\subsection{Additional properties}

In our final section on the general theory of jump operators, we would like to emphasize how they can contribute to the development of a categorical framework for descriptive set theory. The observations in this section are closely related to recent work initiated by A. Pauly on synthetic descriptive set theory \cite{pauly2012talk, pauly_debrecht2013}. 

Let $\cal S=\{\bot,\top\}$ be the Sierpinski space and let ${\bf 2} = \{0,1\}$ be the discrete two point space. It is well known that there is a bijection between the open (resp., clopen) subsets of a topological space $X$ and the continuous functions from $X$ to $\cal S$ (resp., $\bf 2$). In the same manner, there is an obvious bijection between $\bsigma 2(X)$ and the set of $j_{\bsigma 2}$-realizable functions $\chi \colon X \to \cal S$. Furthermore, $\bdelta 2(X)$ is in bijective correspondence with the set of $j_{\bsigma 2}$-realizable functions $\chi \colon X \to {\bf 2}$.

In general, given an arbitrary jump operator $j$ and a represented space $X$, we can define ${\bf \Sigma}_j(X)$ to be the set of $j$-realizable functions from $X$ into $\cal S$, and define ${\bf \Delta}_j(X)$ to be the set of $j$-realizable functions from $X$ into $\bf 2$. Thus, each jump operator $j$ determines a ``$j$-decideable'' class ${\bf \Delta}_j(X)$ of subsets of $X$ and a ``$j$-semi-decideable'' class ${\bf \Sigma}_j(X)$.\footnote{Note that ${\bf \Delta}_j$ and ${\bf \Sigma}_j$ will completely coincide for some jump operators, such as $j_\Delta$.}

It is well known that the category of represented spaces and continuously realizable (total) functions is cartesian closed (see \cite{bauer}, for example). Given a represented space $Y$ and a jump operator $j$, recall that $j(Y)$ denotes the represented space obtained by composing the representation with $j$ (this is the image of $Y$ under the endofunctor determined by $j$). Then for any pair of represented spaces $X$ and $Y$, the exponential object $j(Y)^X$ is the natural candidate for the represented space of $j$-realizable functions from $X$ to $Y$. In particular, $j(\cal S)^X$ corresponds to ${\bf \Sigma}_j(X)$ and $j({\bf 2})^X$ corresponds to ${\bf \Delta}_j(X)$.

We can therefore define notions such as ``$\bsigma 2$-set'' on an arbitrary represented space $X$, and we can interpret the set of $\bsigma 2$-sets as a new represented space. This can be done even when it is impossible to interpret $X$ as a topological space in any natural way.

What kind of a space is $\bsigma 2(\bsigma 2(X))$? Note that $j_{\bsigma 2}(\cal S)$ is isomorphic to the Sierpinski space with the \emph{total} representation $\rho\colon\baire\to\cal S$ sending $\xi\in\baire$ to $\top$ if and only if $(\exists n)(\forall m)[\xi(\langle n,m\rangle) = 1]$. Thus, $\bsigma 2(X)$ represents in a sense the family of ${\Sigma}^0_2$-predicates on $X$, and $\bsigma 2(\bsigma 2(X))$ is the second-order object corresponding to the family of ${\Sigma}^0_2$-predicates on the ${\Sigma}^0_2$-predicates on $X$. This connection between $\bsigma n(X)$ and ${\Sigma}^0_2$-predicates easily extends to $n>2$. It is the topic of future research to determine what kind of general ``topological'' information can be extracted from spaces like $\bsigma 2(\bsigma 2(X))$.

\section{Levels of discontinuity}

The next part of this paper will be dedicated to characterizing $j_{\Delta}$- and $j_{\alpha}$-realizability ($1\leq \alpha <\omega_1$) for functions between arbitrary countably based $T_0$-spaces.  

A characterization of $j_\Delta$-realizability for functions on $\baire$ has already been given by A. Andretta \cite{andretta}. In addition, L. Motto Ros \cite{ros_thesis} has independently investigated a notion related to $j_\alpha$-realizability on metric spaces. However, the extension of the theory to arbitrary countably based $T_0$-spaces that we provide here appears to be new.

In the following sections, we will assume that all represented spaces are countably based $T_0$-topological spaces with admissible representations. Recall from \cite{schroder, weihrauch} that a representation $\rho\colon\subseteq \baire \to X$ to a topological space $X$ is \emph{admissible} if $\rho$ is continuous and for any continuous $f\colon\subseteq \baire\to X$ there is continuous $F\colon \baire \to \baire$ such that $f=\rho\circ F$. It is well known that a function $f\colon X\to Y$ between admissibly represented spaces is continuously realizable if and only if it is continuous.

\subsection{Characterization of $j_\Delta$-realizability}

A total function $f\colon X\to Y$ is \emph{$\bdelta 2$-piecewise continuous} if and only if there is a family $\{A_i\}_{i\in\omega}$ of sets in $\bdelta 2(X)$ such that $X=\bigcup_{i\in\omega}A_i$ and $f|_{A_i}\colon A_i \to Y$, the restriction of $f$ to $A_i$, is continuous for all $i\in\omega$. 

Let $\omega_{\infty}$ be the one point compactification of the natural numbers, with $\infty$ the point at infinity. Recall that a function $\xi \colon \omega_\infty \to X$ is continuous if and only if the sequence $(\xi(i))_{i\in\omega}$ converges to $\xi(\infty)$ in $X$. Given a continuous function $\xi \colon \omega_\infty \to X$ and $S\subseteq X$, we say that \emph{$\xi$ is eventually in $S$} if and only if $\xi(\infty)\in S$ and $\xi(m)\in S$ for all but finitely many $m\in\omega$. We will say that \emph{$\xi$ is eventually equal to $x$} for some $x\in X$ if $\xi$ is eventually in the singleton set $\{x\}$, and in this case we will also say that \emph{$\xi$ is eventually constant}.

Assuming, as we do, that $X$ and $Y$ are countably based,  a function $f\colon X\to Y$ is $\bdelta 2$-piecewise continuous if and only if there is a $\bdelta 2$-measurable function $\iota\colon X\to \omega$ such that for any continuous function $\xi \colon \omega_\infty \to X$, if $\iota\circ \xi$ is eventually constant then $f\circ \xi$ is continuous. Converting from the previous definition to this definition only requires the equivalence between continuity and sequential continuity for countably based spaces, and the generalized $\bsigma 2$-reduction principle which allows us to convert a $\bsigma 2$-covering into a $\bsigma 2$-partitioning. We will call the function $\iota\colon X\to \omega$ above a \emph{${\bdelta 2}$-indexing function} for $f$. For example, the function $\iota_\Delta\colon dom(j_\Delta)\to \omega$ that maps each $\langle \xi_n\rangle_{n\in\omega}\in dom(j_\Delta)$ to the least $n\in\omega$ satisfying $(\forall m\geq n)[\xi_m = \xi_n]$ is a ${\bdelta 2}$-indexing function for $j_{\Delta}$.

The next theorem generalizes a result by A. Andretta \cite{andretta}.

\begin{theorem}\label{thrm:jdelta_char}
Let $f\colon X\to Y$ be a function between admissibly represented countably based $T_0$-spaces. Then $f$ is $j_\Delta$-realizable if and only if $f$ is $\bdelta 2$-piecewise continuous.
\end{theorem}
\begin{proof}
Let $\rho_X$ be the admissible representation for $X$ and $\rho_Y$ the admissible representation for $Y$. We can assume without loss of generality that $\rho_X$ is an open map and has Polish fibers (i.e. $\rho_X^{-1}(x)$ is Polish for each $x\in X$), and similarly for $\rho_Y$.

Assume $F\colon\subseteq \baire\to\baire$ $j_\Delta$-realizes $f$. Then $\iota' = \iota_\Delta \circ F$ is a $\bdelta 2$-indexing function for $f\circ \rho_X = \rho_Y \circ j_\Delta \circ F$. Since $\iota'$ is $\bdelta 2$-measurable, we can write $\iota'^{-1}(n) = \bigcup_{i\in\omega} A^n_i$ for suitably chosen closed sets $A^n_i$. Let $\{B_i\}_{i \in\omega}$ be a countable basis for $\baire$, and define $U^k_{n,i} = \rho_X(B_k)$ and $V^k_{n,i} = \rho_X(B_k\setminus A^n_i)$. Note that $U^k_{n,i}$ and $V^k_{n,i}$ are open subsets of $X$ by our assumption that $\rho_X$ is an open map. 

We first show that each $x\in X$ is in $U^k_{n,i}\setminus V^k_{n,i}$ for some choice of $k,n,i\in\omega$. Since $\rho_X^{-1}(x) \subseteq \bigcup_{n,i\in\omega} A^n_i$, the Baire category theorem implies some $A^n_i$ must have non-empty interior in $\rho_X^{-1}(x)$. Thus there is some $k\in\omega$ such that $B_k\cap \rho_X^{-1}(x)\not=\emptyset$ and $B_k\cap \rho_X^{-1}(x) \subseteq A^n_i \cap \rho_X^{-1}(x)$. It follows that $x\in U^k_{n,i}\setminus V^k_{n,i}$.

Let $\langle \cdot, \cdot,\cdot \rangle\colon \omega^3 \to \omega$ be a bijection, and define $\iota \colon X\to \omega$ so that $\iota(x) = \langle k, n, i \rangle$, where $\langle k,n,i\rangle$ is the least number satisfying $x\in U^k_{n,i}\setminus V^k_{n,i}$. It is immediate that $\iota$ is $\bdelta 2$-measurable.

Let $\xi\colon \omega_\infty \to X$ be a continuous function such that $\iota\circ \xi$ is eventually constant. The admissibility of $\rho_X$ implies there is continuous $\xi'\colon \omega_\infty\to\baire$ such that $\xi = \rho_X\circ \xi'$. Assume $(\iota\circ\xi)(\infty) = \langle k,n,i\rangle$. Then $\xi$ is eventually in $\rho_X(B_k)\setminus \rho_X(B_k\setminus A^n_i)$, hence $\xi'$ is eventually in $B_k \cap A^n_i$, and it follows that $\iota' \circ \xi'$ is eventually equal to $n$. Since $\iota'$ is a $\bdelta 2$-indexing function for $f\circ \rho_X$, it follows that $f\circ \xi = f\circ \rho_X\circ \xi'$ is continuous. Therefore, $\iota$ is a $\bdelta 2$-indexing function for $f$, and we have proven that $f$ is $\bdelta 2$-piecewise continuous.

For the converse, let $\iota\colon X\to \omega$ be a $\bdelta 2$-indexing function for $f$. Then $\iota' = \iota \circ \rho_X$ is a $\bdelta 2$-indexing function for $f\circ \rho_X$. We can write $\iota'^{-1}(n) = \bigcup_{i\in\omega} A^n_i$ for suitably chosen closed sets $A^n_i$, and we have that $f\circ\rho_X$ restricted to $A^n_i$ is continuous. By the admissibility of $\rho_Y$, there is continuous $F^n_i\colon \subseteq \baire\to\baire$ that realizes the restriction of $f\circ \rho_X$ to $A_n^i$. By relabeling, we can assume that $\{A_i\}_{i\in\omega}$ is a family of closed sets covering the domain of $f\circ\rho_X$, and $F_i$ is a continuous realizer for the restriction of $f\circ\rho_X$ to $A_i$.

The most intuitive way to explain how to ``glue'' together the continuous realizers $F_i$ into a single $j_\Delta$-realizer $F$, is to define an algorithm for a Type Two Turing Machine that computes $F$ (possibly with access to some oracle). This description will also help clarify the connections between limit computing with finite mind changes and the $j_\Delta$ jump operators. The reader should consult \cite{weihrauch} for more on Type Two Turing Machines, and \cite{ziegler} for an intuitive description of computing with finite mind changes.

The realizer $F$ first initializes a pointer $p := 0$, and begins reading in the input $\xi\in\baire$. While reading in the input, $F$ attempts to write to its output tape (an encoding of) an infinite sequence of copies of the output of $F_p(\xi)$. In parallel, $F$ will try to determine whether or not $\xi$ really is in $A_p$. If $\xi$ is \emph{not} in $A_p$, then this will be observed after reading in some finite prefix of $\xi$ because $A_p$ is a closed set. In such a case, $F$ will increment the pointer $p:=p+1$, and then resume outputting copies of the output of $F_p(\xi)$ and testing whether $x\in A_p$ for the updated value of the pointer $p$. 

When $p$ is incremented, it is possible that $F$ has already written some finite prefixes of a finite number of elements of $\baire$ to the output tape. After incrementing the pointer, $F$ will consider these initial guesses to be invalid, and will complete the prefixes it has already written by extending them with infinitely many zeros. This guarantees that $F$ will produce a valid encoding of an infinite sequence of elements of $\baire$ as output.     

Since $\{A_i\}_{i\in\omega}$ covers the domain of $f\circ\rho_X$, after a finite number of ``mind changes'' the pointer $p$ will reach a value such that $\xi\in A_p$, and the pointer will never be modified again afterwards. Since $F_p$ realizes the restriction of $f\circ\rho_X$ to $A_p$, we see that the output of $F$ converges after a finite number of modifications to the desired output.
\end{proof}

\subsection{Characterization of $j_\alpha$-realizability}

In this section we will characterize $j_\alpha$-realizability in terms of a hierarchy of discontinuity levels introduced by P. Hertling \cite{hertling_thesis, hertling1996}.

Recall that a function $f\colon X\to Y$ is continuous at $x\in X$ iff for any neighborhood $V$ of $f(x)$ there is an open neighborhood $U$ of $x$ such that $f(U)\subseteq V$. If $f$ is not continuous at $x$ then $f$ is \emph{discontinuous} at $x$.

\begin{definition}[P. Hertling \cite{hertling_thesis, hertling1996}]\label{def:levelofdiscont}
Let $cl(\cdot)$ be the closure operator on $X$ and let $f\colon X\to Y$ be a function. For each ordinal $\alpha$, define $\cal L_{\alpha}(f)$ recursively as follows:
\begin{enumerate}
\item
$\cal L_0(f)=X$
\item
$\cal L_{\alpha+1}(f)=cl(\{ x\in \cal L_{\alpha}(f)\,|\, f|_{\cal L_{\alpha}(f)} \mbox{ is discontinuous at } x\})$
\item
If $\alpha$ is a limit ordinal, then $\cal L_{\alpha}(f)=\bigcap_{\beta<\alpha}\cal L_{\beta}(f)$.
\end{enumerate}
The \emph{level of $f$}, denoted $Lev(f)$, is defined as $Lev(f)=\min\{\alpha\,|\, \cal L_{\alpha}(f)=\emptyset\}$ if there exists $\alpha$ such that $\cal L_{\alpha}(f)=\emptyset$, and $Lev(f)=\infty$, otherwise.
\qed
\end{definition}

Note that, assuming that $X$ is countably based, there is some $\alpha<\omega_1$ such that $\cal L_{\alpha}(f) = \cal L_{\alpha+1}(f)$. This is because we cannot have a strictly decreasing transfinite sequence of closed sets in $X$ with non-countable order type (see Theorem 6.9 in \cite{kechris}). In particular, if $Lev(f)\not=\infty$ then $Lev(f)<\omega_1$ when the domain is countably based.

The next definition will provide a convenient characterization of $Lev(\cdot)$ in terms of ``piecewise continuity''.

\begin{definition}\label{def:dalphapiecewise}
For each ordinal $\alpha$ ($1\leq \alpha <\omega_1$), a total function $f\colon X\to Y$ is \emph{$\bdiff \alpha$-piecewise continuous} if and only if there is a family $\{U_{\beta}\}_{\beta< \alpha}$ of open subsets of $X$ such that $X=\bigcup_{\beta< \alpha}\bdiff \alpha(U_\beta)$ and $f|_{\bdiff \alpha(U_\beta)}\colon \bdiff \alpha(U_\beta) \to Y$ is continuous for all $\beta<\alpha$, where we define $\bdiff \alpha(U_\beta)=U_{\beta}\setminus \bigcup_{\gamma<\beta}U_{\gamma}$.
\qed
\end{definition}

The following theorem shows that Definitions \ref{def:levelofdiscont} and \ref{def:dalphapiecewise} describe equivalent hierarchies of discontinuity.

\begin{theorem}\label{thrm:jalpha_char}
Let $X$ and $Y$ be non-empty countably based $T_0$ spaces, and $f\colon X\to Y$ a function. Then $Lev(f)=\alpha$ ($\alpha\not=\infty$) if and only if $f$ is $\bdiff {\alpha}$-piecewise continuous and $f$ is not $\bdiff {\beta}$-piecewise continuous for any $\beta<\alpha$.
\end{theorem}
\begin{proof}
We divide the proof into two parts. In Part 1, we show that if $Lev(f)=\alpha$ then $f$ is $\bdiff {\alpha}$-piecewise continuous. In Part 2, we show that if $f$ is $\bdiff {\alpha}$-piecewise continuous then $Lev(f)\leq \alpha$. The theorem clearly follows from these two claims.

(\emph{Part 1}): First assume that $Lev(f)=\alpha$. Clearly, $\alpha\geq 1$ because $\cal L_0(f)=X\not=\emptyset$. For $\beta<\alpha$, define $U_\beta = X\setminus \cal L_{\beta+1}(f)$. Clearly $U_\beta$ is open. Note that 
\begin{eqnarray*}
\bigcup_{\gamma<\beta}U_{\gamma} &=& \bigcup_{\gamma<\beta}\big(X\setminus \cal L_{\gamma+1}(f)\big)\\
		&=& X\setminus \bigcap_{\gamma<\beta}\cal L_{\gamma+1}(f)\\
		&=& X\setminus \cal L_{\beta}(f)
\end{eqnarray*}
where the last equality holds when $\beta$ is a limit ordinal by definition of $\cal L_{\beta}(f)$ and holds when $\beta$ is a successor by the fact that $\{\cal L_{\gamma+1}\}_{\gamma<\beta}$ is a decreasing sequence that ends with $\cal L_{\beta}$. It follows that
\begin{eqnarray*}
\bdiff {\alpha}(U_\beta) &=& U_{\beta}\setminus \bigcup_{\gamma<\beta}U_{\gamma}\\
				&=& \big(X\setminus \cal L_{\beta+1}(f)\big)\setminus \big(X \setminus \cal L_{\beta}(f)\big)\\
				&=& \cal L_{\beta}(f) \setminus \cal L_{\beta+1}(f).
\end{eqnarray*}
We first show that $X=\bigcup_{\beta<\alpha}\bdiff {\alpha}(U_\beta)$. For $x\in X$, let $\beta_x = \min\{ \beta \,|\, x\not\in \cal L_{\beta}(f)\}$. Since $\cal L_{\alpha}(f)=\emptyset$ by assumption, we have that $\beta_x$ is defined and $\beta_x\leq \alpha$. It is also clear that $\beta_x$ is a successor ordinal, because if $\beta_x$ was a limit ordinal then $x\in\cal L_{\gamma}(f)$ for all $\gamma<\beta$ (by our choice of minimal $\beta_x$) hence $x\in \cal L_{\beta_x}(f)$ (by definition of $\cal L_{\beta_x}$ for limit $\beta_x$), a contradiction. Therefore, $\beta_x = \gamma_x+1$ for some ordinal $\gamma_x<\alpha$. It follows that $x\in \bdiff {\alpha}(U_{\gamma_x})$, hence $X=\bigcup_{\beta<\alpha}\bdiff {\alpha}(U_\beta)$.

It only remains to show that $f|_{\bdiff {\alpha}(U_\beta)}$ is continuous for all $\beta<\alpha$. Assume for a contradiction that $f|_{\bdiff {\alpha}(U_\beta)}$ is discontinuous at some point $x$. Since $\bdiff {\alpha}(U_\beta)$ is a subspace of $\cal L_{\beta}(f)$, it must be the case that $f|_{\cal L_{\beta}(f)}$ is also discontinuous at $x$. Therefore, $x\in \cal L_{\beta+1}(f)$, contradicting $x\in \bdiff {\alpha}(U_\beta)$. Thus, $f|_{\bdiff {\alpha}(U_\beta)}$ is continuous for all $\beta<\alpha$. 

(\emph{Part 2}): Let $\{U_{\beta}\}_{\beta< \alpha}$ be open subsets of $X$ such that $X=\bigcup_{\beta< \alpha}\bdiff \alpha(U_\beta)$ and $f|_{\bdiff \alpha(U_\beta)}\colon \bdiff \alpha(U_\beta) \to Y$ is continuous for all $\beta<\alpha$. We can assume without loss of generality that $\bigcup_{\gamma<\beta}U_\gamma \subseteq U_\beta$ for all $\beta<\alpha$.

We claim that $\cal L_{\beta+1}(f)\subseteq X\setminus U_\beta$ for all $\beta<\alpha$. The case $\beta=0$ is easy, so assume that $\beta>0$ and the claim holds for all $\gamma<\beta$. First note that, since $f|_{\bdiff \alpha(U_\beta)}=f|_{U_\beta \setminus \bigcup_{\gamma<\beta} U_\gamma}$ is continuous by assumption, and since $U_\beta \setminus \bigcup_{\gamma<\beta} U_\gamma$ is an open subspace of $X \setminus \bigcup_{\gamma<\beta} U_\gamma$, $f|_{X\setminus \bigcup_{\gamma<\beta} U_\gamma}$ is continuous at every $x\in U_\beta$. Now assume that $f|_{\cal L_{\beta}(f)}$ is discontinuous at $x$. By induction hypothesis,
\[\cal L_\beta(f) = \bigcap_{\gamma<\beta}\cal L_{\gamma+1}(f)\subseteq \bigcap_{\gamma<\beta} X\setminus U_\gamma= X\setminus \bigcup_{\gamma<\beta} U_\gamma\]
so it follows that $f|_{X\setminus \bigcup_{\gamma<\beta} U_\gamma}$ is discontinuous at $x$. Therefore, $x\not\in U_\beta$. Hence,
\[\{ x\in \cal L_{\beta}(f)\,|\, f|_{\cal L_{\beta}(f)} \mbox{ is discontinuous at } x\}\subseteq X\setminus U_\beta\]
and it follows that $\cal L_{\beta+1}(f)\subseteq X\setminus U_\beta$ because $X\setminus U_\beta$ is closed. This concludes the proof of the claim.

Since $\{X\setminus U_\beta\}_{\beta<\alpha}$ is a decreasing sequence of closed sets and $\bigcap_{\beta<\alpha}(X\setminus U_\beta) = \emptyset$, the claim implies that $\cal L_{\alpha}(f)=\emptyset$, hence $Lev(f)\leq \alpha$.
\end{proof}

For each countable ordinal $\alpha$, we let $\alpha^{op}$ denote the topological space whose points are the ordinals less than $\alpha$ and whose open sets are generated from the sets $\downarrow\!\beta = \{\gamma \,|\, \gamma\leq \beta\}$ for each $\beta<\alpha$.

Let $f\colon X\to Y$ be a function between countably based spaces, and let $\alpha$ be a countable ordinal. An \emph{$\alpha$-indexing function} for $f$ is a continuous function $\iota\colon X\to \alpha^{op}$ such that for any continuous function $\xi \colon \omega_\infty \to X$, if $\iota\circ \xi$ is eventually constant then $f\circ \xi$ is continuous. 

The existence of an $\alpha$-indexing function is a necessary and sufficient condition for a function to be $\bdiff \alpha$-piecewise continuous. If $f$ is $\bdiff {\alpha}$-piecewise continuous, then the function $\iota$ mapping $\bdiff \alpha(U_\beta)$ to $\beta$ is an $\alpha$-indexing function for $f$. Conversely, if $\iota$ is an $\alpha$-indexing function for $f$, then defining $U_\beta = \iota^{-1}(\downarrow\!\beta)$ demonstrates that $f$ is $\bdiff \beta$-piecewise continuous.

The function $\iota_\alpha \colon dom(j_\alpha) \to \alpha^{op}$, defined as mapping $\langle\langle \beta_n\rangle_\alpha\diamond \xi_n \rangle_{n\in\omega}$ to $\min\{\beta_n\,|\,n\in\omega\}$, is an $\alpha$-indexing function for the jump operator $j_\alpha$.

We can now completely characterize $j_\alpha$-realizability for functions between countably based spaces.

\begin{theorem}
Let $f\colon X\to Y$ be a function between admissibly represented countably based $T_0$-spaces and let $\alpha$ be a countable ordinal. Then $f$ is $j_\alpha$-realizable if and only if $Lev(f)\leq \alpha$ if and only if $f$ is $\bdiff \alpha$-piecewise continuous.
\end{theorem}
\begin{proof}
As before, we assume without loss of generality that the admissible representations $\rho_X$ and $\rho_Y$ are open maps with Polish fibers.

Assume $F\colon \subseteq \baire\to\baire$ $j_\alpha$-realizes $f$. Then $\iota' = \iota_\alpha \circ F$ is an $\alpha$-indexing function for $f\circ \rho_X = \rho_Y\circ j_\alpha\circ F$. Let $U_\beta = \rho_X(\iota'^{-1}(\downarrow\!\beta))$ for $\beta<\alpha$. Clearly, $U_\beta$ is an open subset of $X$ because $\rho_X$ is an open map. Finally, define $\iota\colon X \to \alpha^{op}$ so that $x\mapsto \min\{\beta\,|\, x\in U_\beta\}$. It is easy to see that $\iota$ is a well-defined total function. For each ordinal $\beta<\alpha$, $\iota^{-1}(\downarrow\!\beta)=\bigcup_{\gamma\leq\beta} U_\gamma$, hence $\iota$ is continuous.

Let $\xi\colon \omega_\infty \to X$ be a continuous function such that $\iota\circ \xi$ is eventually constant. The admissibility of $\rho_X$ implies there is continuous $\xi'\colon \omega_\infty\to\baire$ such that $\xi = \rho_X\circ \xi'$. Assume $(\iota\circ\xi)(\infty) = \beta$. Then $\xi$ is eventually in $\bdiff \alpha(U_\beta)$, hence $\xi'$ is eventually in $\iota'^{-1}(\downarrow\!\beta)\setminus \bigcup_{\gamma<\beta}\iota'^{-1}(\downarrow\!\gamma)$. It follows that $\iota'\circ\xi'$ is eventually equal to $\beta$. Since $\iota'$ is an $\alpha$-indexing function for $f\circ \rho_X$, it follows that $f\circ \xi = f\circ \rho_X\circ \xi'$ is continuous. Therefore, $\iota$ is an $\alpha$-indexing function for $f$, and we have proven that $f$ is $\bdiff \alpha$-piecewise continuous.

For the converse, we will use the same method as in the proof of Theorem \ref{thrm:jdelta_char} and define an oracle Type Two Turing Machine that computes a $j_\alpha$-realizer for $f$. Let $\iota\colon X\to \omega$ be an $\alpha$-indexing function for $f$. Then $\iota' = \iota \circ \rho_X$ is an $\alpha$-indexing function for $f\circ \rho_X$. Let $U_\beta = \iota'^{-1}(\downarrow\!\beta)$ for $\beta<\alpha$, and we have that $f\circ\rho_X$ restricted to $\bdiff \alpha(U_\beta)$ is continuous. By the admissibility of $\rho_Y$, there is continuous $F_\beta\colon \subseteq \baire\to\baire$ that realizes the restriction of $f\circ \rho_X$ to $\bdiff \alpha(U_\beta)$.

Our algorithm is as follows. Begin reading in the input $\xi\in\baire$, and search in parallel for $\beta<\alpha$ such that $\xi \in U_\beta$. Such a $\beta$ can be found after reading in a finite prefix of $\xi$ because each $U_\beta$ is open and the $U_\beta$ cover the domain of $f\circ\rho_X$. The algorithm then initializes an ordinal counter $\hat{\beta} := \beta$ and attempts to write to the output tape an infinite sequence of copies of the element $\langle  \hat{\beta} \rangle_\alpha\diamond F_{\hat{\beta}}(\xi)$. While outputting the copies of $\langle  \hat{\beta} \rangle_\alpha\diamond F_{\hat{\beta}}(\xi)$ the algorithm continues to search for some $\gamma<\hat{\beta}$ such that $\xi \in U_\gamma$. If such a $\gamma$ is ever found, then the algorithm sets $\hat{\beta} := \gamma$ and begins outputting an infinite sequence of copies of $\langle  \hat{\beta} \rangle_\alpha\diamond F_{\hat{\beta}}(\xi)$ for the new value of $\hat{\beta}$. It is easy to see that such an algorithm computes a $j_\alpha$-realizer for $f$.
\end{proof}

\section{Examples and applications}

In this last section of this paper we provide a few examples and applications of $j_\Delta$ and $j_\alpha$-realizability. 

\subsection{The Difference Hierarchy}

Given a jump operator $j$ and a represented space $X$, recall that ${\bf \Delta}_j(X)$ is the set of $j$-realizable functions from $X$ into the discrete two point space ${\bf 2}=\{0,1\}$. In this section, we will show that ${\bf \Delta}_{j_\alpha}(X)$ ($1\leq\alpha<\omega_1$) correspond to the ambiguous levels of the difference hierarchy when $X$ is a countably based space. 

\begin{definition}
Any ordinal $\alpha$ can be expressed as $\alpha=\beta+n$, where $\beta$ is a limit ordinal or $0$, and $n<\omega$. We say that $\alpha$ is \emph{even} if $n$ is even, and \emph{odd}, otherwise. For any ordinal $\alpha$, let $r(\alpha)=0$ if $\alpha$ is even, and $r(\alpha)=1$, otherwise. For any ordinal $\alpha$, define
\[\mybf D_{\alpha}(\{A_{\beta}\}_{\beta<\alpha})=\bigcup \{A_{\beta}\setminus \big(\bigcup_{\gamma<\beta}A_{\gamma}\big) \,|\, \beta<\alpha,\, r(\beta)\not=r(\alpha)\},\]
where $\{A_{\beta}\}_{\beta<\alpha}$ is a sequence of sets such that $A_{\gamma}\subseteq A_{\beta}$ for all $\gamma < \beta < \alpha$. 

For any topological space $X$ and ordinal $\alpha$, define ${\bf \Sigma}^{-1}_{\alpha}(X)$ to be the set of all sets $\mybf D_{\alpha}(\{U_{\beta}\}_{\beta<\alpha})$, where $\{U_{\beta}\}_{\beta<\alpha}$ is an increasing sequence of open subsets of $X$. 
\qed
\end{definition}

The following connection with the difference hierarchy has already been observed by both P. Hertling and V. Selivanov, so we omit the proof.

\begin{proposition}[see \cite{selivanov2007}]
If $X$ is a countably based $T_0$-space and $1\leq \alpha <\omega_1$, then a total function $f\colon X\to {\bf 2}$ is $j_\alpha$-realizable if and only if both $f^{-1}(1)$ and $f^{-1}(0)$ are in ${\bf \Sigma}^{-1}_{\alpha}(X)$.
\qed
\end{proposition}

\subsection{Cantor-Bendixson Rank}

A \emph{limit point} of a topological space is a point that is not isolated, i.e. a point $x$ such that for every open $U$ containing $x$ there is $y\in U$ distinct from $x$. A space is \emph{perfect} if all of its points are limit points.

\begin{definition}[see \cite{kechris}]
For any topological space $X$, let
\[X' = \{x\in X\,|\, x\mbox{ is a limit point of } X\}.\]
For ordinal $\alpha$, define $X^{(\alpha)}$ recursively as follows:
\begin{enumerate}
\item
$X^{(0)}=X$,
\item
$X^{(\alpha+1)}=(X^{(\alpha)})'$,
\item
If $\alpha$ is a limit ordinal, then $X^{(\alpha)}=\bigcap_{\beta<\alpha}X^{(\beta)}$.
\end{enumerate}
If $X$ is countably based, then there is a least countable ordinal $\alpha_0$ such that $X^{(\alpha)}=X^{(\alpha_0)}$ for all $\alpha\geq \alpha_0$. Such $\alpha_0$ is called the \emph{Cantor-Bendixson rank} of $X$, and is denoted $|X|_{CB}$. We also let $X^{\infty}=X^{(|X|_{CB})}$, which is a perfect subset of $X$.
\qed
\end{definition}

Assuming $X$ is countably based, $X\setminus X^{\infty}$ must be countable. This is because for every $x\in (X\setminus X')$ there must be a (basic) open $U$ containing $x$ and no other elements of $X$, so $(X\setminus X')$ must be countable.

Let $\omega_\bot=\omega\cup\{\bot\}$ be such that $\{n\}$ is open for $n\in \omega$ and the only open set containing $\bot$ is $\omega_\bot$ itself. Given countably based $X$, define $p {\colon} X\to \omega_\bot$ so that $p(x)=\bot$ for $x\in X^{\infty}$ and $p$ restricted to the elements of $X\setminus X^{\infty}$ is injective into $\omega$. 

The following is a generalization of a result by Luo and Schulte \cite{luo_schulte} concerning ordinal mind-change complexity of inductive inference (see also \cite{debrecht_etal_2, debrecht_etal_3}).

\begin{proposition}
For any countably based space $X$, $p {\colon} X\to \omega_\bot$ is $j_\alpha$-realizable, where
\begin{enumerate}
\item
$\alpha= |X|_{CB}$ if $X^{\infty}=\emptyset$ or $|X|_{CB}$ is a successor ordinal
\item
$\alpha= |X|_{CB}+1$ if $X^{\infty}\not=\emptyset$ and $|X|_{CB}$ is a limit ordinal
\end{enumerate}
\end{proposition}
\begin{proof}
If $X^{\infty}=\emptyset$, then define $U_\beta = X\setminus (X^{(\beta)})'$ for $\beta<|X|_{CB}$. Letting $\alpha = |X|_{CB}$, we see that $\bdiff \alpha(U_\beta) = X^{(\beta)}\setminus (X^{(\beta)})'$, which is the set of isolated points of $X^{(\beta)}$, hence a discrete subspace of $X$.

If $X^{\infty}\not=\emptyset$ but $|X|_{CB}=\gamma+1$, then set $U_\beta = X\setminus (X^{(\beta)})'$ for $\beta<\gamma$ and $U_\gamma = X$. Letting $\alpha = |X|_{CB}$, $\bdiff \alpha(U_\beta) = X^{(\beta)}\setminus (X^{(\beta)})'$ for $\beta<\gamma$ and $\bdiff \alpha(U_\gamma) = X^{\infty}$.

If $X^{\infty}\not=\emptyset$ and $|X|_{CB}$ is a limit ordinal, then set $U_\beta = X\setminus (X^{(\beta)})'$ for $\beta<|X|_{CB}$ and $U_{|X|_{CB}} = X$. Letting $\alpha = |X|_{CB}+1$, $\bdiff \alpha(U_\beta) = X^{(\beta)}\setminus (X^{(\beta)})'$ for $\beta<|X|_{CB}$ and $\bdiff \alpha(U_{|X|_{CB}}) = X^{\infty}$.

In all three of the above cases, it is easy to see that $U_\beta$ $(\beta<\alpha)$ is open, $X=\bigcup_{\beta<\alpha}\bdiff \alpha(U_\beta)$, and the corresponding restrictions of $p$ are continuous.
\end{proof}

\subsection{Hilbert's basis theorem}

Let $(R,+,\cdot)$ be a commutative ring. A subset $I\subseteq R$ is an \emph{ideal} if and only if $(I,+)$ is a subgroup of $(R,+)$ and $(\forall x\in I)(\forall r\in R)[ x\cdot r \in I ]$. A ring is \emph{Noetherian} if and only if it does not have an infinite strictly ascending chain of ideals.

Given a ring $R$, we let $R[x_1,\ldots,x_n]$ denote the ring of polynomials with coefficients in $R$ and $n$ indeterminates $x_1,\ldots,x_n$. 

A famous theorem by David Hilbert states that if $R$ is a Noetherian ring then $R[x_1,\ldots,x_n]$ is also Noetherian. Hilbert's proof was non-constructive, and was initially criticized by Paul Gordan with the famous quote ``Das ist nicht Mathematik. Das ist Theologie.''

Here we quantify one aspect of the ``non-constructiveness'' of the basis theorem in terms of the level of discontinuity of converting an enumeration of an ideal into a Gr\"{o}bner basis for the ideal. Our approach is much in the same spirit as V. Brattka's project to quantify the non-computability of mathematical theorems in terms of their Weihrauch degrees (see, for example, \cite{brattka_gherardi, brattka_gherardi2,brattka_debrecht_pauly}). Our contribution is only in the way that we formalize the problem, and our main result is essentially a reformulation of results on Hilbert's basis theorem by S. Simpson \cite{simpson1988} in the context of reverse mathematics, and by F. Stephan and Y. Ventsov \cite{stephan_ventsov} in the context of inductive inference.

Let $\bb Q$ be the ring of rational numbers, and let $\cal I_n$ be the set of ideals of the polynomial ring $\bb Q[x_1,\ldots, x_n]$. If we encode the elements of $\bb Q[x_1\ldots,x_n]$ as elements of $\omega$, then each element of $\baire$ can be interpretted as an infinite sequence of elements of $\bb Q[x_1\ldots,x_n]$. We will interpret $\cal I_n$ as a represented space with the representation $\rho\colon\subseteq \baire\to \cal I_n$ that maps each enumeration of an ideal $I\in \cal I_n$ to $I$. This representation is admissible with respect to the topology on $\cal I_n$ generated by $\uparrow\!\langle r\rangle=\{I\in\cal I_n \,|\, r\in I\}$, where $r$ varies over elements of $\bb Q[x_1\ldots,x_n]$. Note that this topology is very far from being Hausdorff.

Let ${\bf G}_n$ be the set of finite subsets of $\bb Q[x_1\ldots,x_n]$. We will think of ${\bf G}_n$ as the set of Gr\"{o}bner bases for ideals in $\cal I_n$ (for some predefined monomial order). We think of each Gr\"{o}bner basis in ${\bf G}_n$ as being represented by a finite terminated string, hence ${\bf G}_n$ carries the discrete topology. 

Let $f_n\colon \cal I_n \to {\bf G}_n$ be the function that maps each $I\in \cal I_n$ to its unique Gr\"{o}bner basis. Intuitively, $f_n$ embodies the problem of converting an infinite enumeration of an ideal of $\bb Q[x_1\ldots,x_n]$ into a finite Gr\"{o}bner basis for the ideal. 

The next theorem immediately follows from work by S. Simpson \cite{simpson1988} and F. Stephan and Y. Ventsov \cite{stephan_ventsov}, so we omit the proof.

\begin{theorem}
The functions $f_n\colon \cal I_n \to {\bf G}_n$ are $j_{\omega^n}$-realizable for each $n\in\omega$. In fact, $Lev(f_n)=\omega^n$.
\end{theorem}

Hilbert's basis theorem holds for all $n\in\omega$, so it is natural to consider the function $\forall_n f_n$ corresponding to universal quantification over $\omega$. The most natural interpretation for such a function is to simply take the disjoint union of all of the $f_n$. Then $\forall_n f_n$ essentially takes some $n\in\omega$ as initial input, and then operates like $f_n$ thereafter. It is easy to see that $Lev(\forall_n f_n) = \omega^\omega$, which is consistent with S. Simpson's \cite{simpson1988} characterization of Hilbert's basis theorem.

\subsection{$\bdelta 2$-measurable functions and the Jayne-Rogers theorem}

Recall that a function $f\colon X\to Y$ is \emph{$\bdelta 2$-measurable} if and only if $f^{-1}(U)\in\bdelta 2(X)$ for each open $U\subseteq Y$. Note that this is equivalent to requiring that $f^{-1}(A)\in\bsigma 2(X)$ for each $A\in\bsigma 2(Y)$.

The following is a slight generalization of a theorem by J. E. Jayne and C. A. Rogers \cite{jayne_rogers}. A much simpler proof of the original theorem was given by L. Motto Ros and B. Semmes \cite{ros_semmes, kacena_etal}. The original version of the Jayne-Rogers theorem only applied to functions that had a metrizable domain.

In the following, an \emph{analytic space} is a topological space that has an admissible representation with analytic domain. For countably based $T_0$-spaces, this is easily seen to be equivalent to the space being homeomorphic to an analytic subset of a quasi-Polish space \cite{debrecht2013}. 

\begin{theorem}[Jayne and Rogers]\label{thrm:jaynerogers}
Assume $X$ is an analytic countably based $T_0$-space and $Y$ is a separable metrizable space. Then a function $f\colon X\to Y$ is $\bdelta 2$-measurable if and only if it is $\bdelta 2$-piecewise continuous.
\end{theorem}
\begin{proof}
Let $\rho_X$ be an admissible representation of $X$ with analytic domain. Then $f\circ \rho_X$ is a $\bdelta 2$-measurable function from an analytic metrizable space into a metrizable space, hence $f\circ \rho_X$ is $\bdelta 2$-piecewise continuous by the original Jayne-Rogers theorem \cite{jayne_rogers,ros_semmes, kacena_etal}. It follows that there is a continuous $j_\Delta$-realizer $F$ of $f\circ\rho_X$. Clearly, $F$ is a $j_\Delta$-realizer of $f$, hence $f$ is $\bdelta 2$-piecewise continuous. 
\end{proof}

A natural question is how much the constraints on the domain and codomain can be relaxed. Based on S. Solecki's work in \cite{solecki1994} and its applications to the Jayne-Roger's theorem \cite{solecki1998} (but see also \cite{kacena_etal}), we conjecture that it is consistent with ZFC to allow the domain to be ${\bf \Sigma}^1_2$ and possibly any other level of the projective hierarchy. With respect to the codomain, however, the following example shows that some separation axiom, such as regularity\footnote{Luca Motto Ros has pointed out to the author that the proof in \cite{kacena_etal} suggests regularity is a sufficient criterion on the codomain for Theorem \ref{thrm:jaynerogers} to hold, even in the absence of separability and metrizability.}, is required.

\begin{example}\label{ex:bdelta2_meas_no_piecewise}
We consider two topologies on the ordinal $\omega+1 = \{0,1,2,\ldots, \omega\}$. The first topology, $\tau_1$, is the Scott-topology, and is generated by the open sets ${\uparrow\!n} = \{ \beta \in\omega+1\,|\, n\leq \beta\}$ for $n<\omega$. The second topology, $\tau_2$, is defined so that a non-empty subset $U\subseteq \omega+1$ is open if and only if $\omega\in U$ and all but finitely many $n\in\omega$ are in $U$. The topological space $(\omega+1, \tau_1)$ is one of the simplest examples of an infinite $\omega$-continuous domain \cite{etal_scott}. The topological space $(\omega+1, \tau_2)$ is homeomorphic to the Zariski topology on the prime spectrum of the ring of integers. Both $\tau_1$ and $\tau_2$ are quasi-Polish topologies on $\omega+1$. Furthermore, in both of these spaces the singleton set $\{\omega\}$ is $\bpi 2$ but \emph{not} $\bsigma 2$, hence these spaces fail the $T_D$-separation axiom of Aull and Thron \cite{aull_thron} (see also \cite{debrecht2013,debrecht_etal_3} for more on the $T_D$-axiom).

The function $f\colon (\omega+1, \tau_1) \to (\omega+1, \tau_2)$, defined to behave as the identity on $\omega+1$, is a $\bdelta 2$-measurable function that is not $\bdelta 2$-piecewise continuous. Since $(\omega+1, \tau_1)$ is quasi-Polish, it follows from \cite{debrecht2013} that it has a total admissible representation $\rho\colon\baire\to (\omega+1,\tau_1)$. As any $j_\Delta$-realizer of $f\circ\rho$ would be a $j_\Delta$-realizer of $f$, we see that the function $f\circ\rho \colon \baire\to (\omega+1, \tau_2)$ is an example of a function with Polish domain which is $\bdelta 2$-measurable but not $\bdelta 2$-piecewise continuous. 
\qed
\end{example}

Let $\cal F$ be a class of functions between (admissibly represented) topological spaces. We will say that a jump operator $j$ \emph{captures} the class $\cal F$ if it holds that $f\in\cal F$ if and only if $f$ is $j$-realizable. Note that such a $j$ must be in $\cal F$ because $j$ is trivially $j$-realizable.

If we let $\cal F$ be the class of $\bdelta 2$-measurable functions with (countably based) analytic domain and metrizable codomain, then the Jayne-Rogers theorem states that $j_\Delta$ captures $\cal F$. However, the example above shows that $j_\Delta$ does not capture the class of $\bdelta 2$-measurable functions between arbitrary countably based $T_0$-spaces.  One might wonder if some other jump operator might capture this larger class, but the following result shows that this is not possible.

\begin{proposition}
There is no jump operator that captures the entire class of $\bdelta 2$-measurable functions between countably based $T_0$-spaces. 
\end{proposition}
\begin{proof}
Assume for a contradiction that $j$ captures the entire class of $\bdelta 2$-measurable functions between countably based $T_0$-spaces. Let $f\colon ({\omega+1}, \tau_1) \to (\omega+1, \tau_2)$ be the $\bdelta 2$-measurable function from Example \ref{ex:bdelta2_meas_no_piecewise}. We can assume that $({\omega+1}, \tau_1)$ has a total admissible representation, hence there is a total continuous $F\colon \baire\to\baire$ that $j$-realizes $f$. The Jayne-Rogers theorem now implies that the $\bdelta 2$-measurable function $j\circ F\colon \baire\to\baire$ has a continuous $j_\Delta$-realizer $F'\colon\baire\to\baire$. However, $F'$ would then be a continuous $j_\Delta$-realizer of $f$, which is a contradiction.
\end{proof}

By applying the same argument to the function $f\circ \rho$ from Example \ref{ex:bdelta2_meas_no_piecewise}, it can be seen that the above proposition holds true even if we further restrict to $\bdelta 2$-measurable functions with Polish domain.
 
\subsection{A generalization of the Hausdorff-Kuratowski theorem}

The Hausdorff-Kuratowski theorem (see \cite{kechris}) states that the difference hierarchy on a Polish space exhausts all of the $\bdelta 2$-sets. The full version of the theorem actually applies to all levels of the Borel hierarchy. It was observed by V. Selivanov \cite{selivanov2004,selivanov2008} that the Hausdorff-Kuratowski theorem holds for some important non-metrizable spaces such as $\omega$-continuous domains. Later it was shown that the full version of the Hausdorff-Kuratowski theorem holds for all quasi-Polish spaces \cite{debrecht2013}.

In addition to extending the Hausdorff-Kuratowski theorem to a more general class of spaces, V. Selivanov has also generalized the theorem from being a classification of sets to a classification of functions \cite{selivanov2004}. In particular, it was observed in \cite{selivanov2004} that each $\bdelta 2$-measurable function $f$ from a Polish space into a finite discrete space will satisfy $Lev(f)=\alpha$ for some $\alpha<\omega_1$. In this section, we will extend this result to show that any $\bdelta 2$-measurable function $f$ from a quasi-Polish space to a separable metrizable space will satisfy $Lev(f)=\alpha$ for some $\alpha<\omega_1$. Given the connections between P. Hertling's levels of discontinuity and the difference hierarchy, our result is a very broad generalization of the Hausdorff-Kuratowski theorem restricted to $\bdelta 2$-sets. L. Motto Ros \cite{ros_thesis} has independently made a similar observation for $\bdelta 2$-measurable functions on metrizable spaces.

As in the original proof of the Hausdorff-Kuratowski theorem, the Baire category theorem plays an important role in our generalized result as well. One version of the Baire category theorem states that if a Polish space is equal to the union of a countable family of closed sets, then one of the closed sets must have non-empty interior. Clearly, the same statement holds for Polish spaces if we replace ``closed'' by either ``$F_\sigma$'' or ``$\bsigma 2$''. However, since the equivalence between $F_\sigma$-sets and $\bsigma 2$-sets breaks down for non-metrizable spaces, the version of the Baire category theorem presented in the following lemma is more appropriate in general. This generalization of the Baire category theorem has already been investigated by R. Heckmann \cite{heckmann} and by V. Becher and S. Grigorieff \cite{becher_grigorieff}.

\begin{lemma}\label{lem:quasipolish_bairecategory}
Assume $X$ is quasi-Polish and $\{A_i\}_{i\in\omega}$ is a family of sets from $\bsigma 2(X)$ such that $X=\bigcup_{i\in\omega} A_i$. Then there is $i\in\omega$ such that $A_i$ has non-empty interior. Equivalently, the intersection of a countable family of dense $\bpi 2$-subsets of a quasi-Polish space is a dense $\bpi 2$-set.
\end{lemma}
\begin{proof}
Let $f\colon \baire \to X$ be an open continuous surjection (see \cite{debrecht2013}) and let $B_i=f^{-1}(A_i)$. Each $B_i$ is a $\bsigma 2$-subset of a metrizable space hence equal to a countable union of closed sets. Since $\baire = \bigcup_{i\in\omega}B_i$, the Baire category theorem for Polish spaces implies there is $i\in\omega$ such that $B_i$ has non-empty interior. It follows that $A_i$ has non-empty interior because $f$ is an open map.
\end{proof}

\begin{theorem}
If $X$ is quasi-Polish and $Y$ is a countably based $T_0$-space, then $f\colon X\to Y$ is $j_\Delta$-realizable if and only if it is $j_\alpha$-realizable for some $\alpha<\omega_1$.
\end{theorem}
\begin{proof}
Assume $f\colon X\to Y$ is $j_\Delta$-realizable. Let $\alpha<\omega_1$ be the least ordinal such that $\cal L_{\alpha}(f) = \cal L_{\alpha+1}(f)$, which exists because $X$ is countably based. Assume for a contradiction that $\cal L_{\alpha}(f)\not=\emptyset$. Clearly, $f|_{\cal L_{\alpha}(f)}$ is $j_\Delta$-realizable, hence there is a $\bdelta 2$-partitioning  $\{A_i\}_{i\in\omega}$ of $\cal L_{\alpha}(f)$ such that the restriction of $f$ to $A_i$ is continuous for each $i\in\omega$. Note that $\cal L_{\alpha}(f)$ is quasi-Polish because it is a closed subset of the quasi-Polish space $X$. Therefore, Lemma \ref{lem:quasipolish_bairecategory} applies and there is $i\in\omega$ such that $A_i$ has non-empty interior relative to $\cal L_{\alpha}(f)$. But then $f|_{\cal L_{\alpha}(f)}$ is continuous on a non-empty open subset of $\cal L_{\alpha}(f)$, contradicting our assumption that $\cal L_{\alpha}(f) = \cal L_{\alpha+1}(f)$. Thus, $\cal L_{\alpha}(f)$ is empty and it follows that $Lev(f)= \alpha<\omega_1$.

The converse holds for all represented spaces because $j_\alpha\leq j_\Delta$.
\end{proof}

Combining the above result with the Jayne-Rogers theorem yields the following generalization of the Hausdorff-Kuratowski theorem.

\begin{theorem}
If $X$ is quasi-Polish and $Y$ is a separable metrizable space, then $f\colon X\to Y$ is $\bdelta 2$-measurable if and only if it is $j_\alpha$-realizable for some $\alpha<\omega_1$.
\qed
\end{theorem}

\section{Conclusions}

Although much of classical descriptive set theory has been extended to arbitrary countably based $T_0$-spaces, it is still a major open problem to understand how descriptive set theory \emph{should} work for non-countably based topological spaces and more general represented spaces. This is a very strange realm, where even singleton sets can have complexity of arbitrarily high rank in the projective hierarchy. 

The approach we have taken here with jump operators provides a general framework, with a nice categorical flavor, for which to extend descriptive set theory to more general mathematical structures. In particular, it raises natural questions concerning the structure and applications of ``higher-order'' descriptive set theoretical objects such as $\bsigma 2(\bsigma 2(X))$.

There is also a strong need for a refined analysis of the categorical logic of the category of represented spaces and realizable functions with closer attention to the ``level'' of the represented spaces. For example, the ``naive Cauchy'' representation of the real numbers in \cite{bauer}, which is obtained via a kind of double negation of the standard Cauchy representation of the reals, happens to be equivalent to the $j_{\bsigma 2}$-jump of the standard Cauchy representation of reals \cite{ziegler}. S. Hayashi \cite{hayashi2006} has also investigated connections between limit-computability and non-constructive principles such as double negation elimination and the excluded middle restricted to certain subclasses of formulae. It would be very interesting to see how these concepts are connected.

\bibliographystyle{amsplain}
\bibliography{myrefs}

\end{document}